\documentclass[twoside]{article}
\usepackage[utf8]{inputenc}

\usepackage{geometry}
\usepackage{amsmath, amssymb, dsfont, amsthm, mathtools, mathrsfs}
\usepackage{tikz-cd}
\usepackage{todonotes}
\usepackage{graphicx}
\usepackage{xcolor}

\usepackage{setspace}
\usepackage{authblk}

\AtEndDocument{%
  \par
  \medskip
  \begin{tabular}{@{}l@{}}%
    \textsc{Address}\\
    School of Mathematics, University of Bristol \\
    \textit{E-mail address}: \texttt{nb14957@bristol.ac.uk}
  \end{tabular}}

\title{\textbf{Counting arcs on hyperbolic surfaces}}
\author{Nick Bell}
\date{}

\usepackage{fancyhdr}
\renewcommand{\subsectionmark}[1]{}

\newtheorem{theorem}{Theorem}

\newtheorem{prop}[theorem]{Proposition}

\newtheorem{lemma}[theorem]{Lemma}

\newtheorem{corollary}[theorem]{Corollary}

\theoremstyle{definition}

\newtheorem*{remark}{Remark}

\newtheorem*{ack}{Acknowledgements}

\newtheorem{example}{Example}

\newcommand{\PMod}{\textup{PMod}}
\newcommand{\Mod}{\textup{Mod}}
\newcommand{\ML}{\mathcal{ML}}

\newcommand{\ee}{{\rm e}}
\newcommand{\de}{\partial}
\newcommand{\mthu}{\mathfrak{m}_{\textup{Thu}}}
\newcommand{\cc}{\mathfrak{c}}
\newcommand{\bb}{\mathfrak{b}}
\newcommand{\ia}{\iota_\alpha}

\renewenvironment{abstract}
               {\list{}{\rightmargin\leftmargin}%
                \item[\hspace{8.5mm}\textbf{Abstract}]\relax}
               {\endlist}

\begin{document}
\maketitle

\begin{abstract}
    We give the asymptotic growth of the number of (multi-)arcs of bounded length between boundary components on complete finite-area hyperbolic surfaces with boundary. Specifically, if $S$ has genus $g$, $n$ boundary components and $p$ punctures, then the number of orthogeodesic arcs in each pure mapping class group orbit of length at most $L$ is asymptotic to $L^{6g-6+2(n+p)}$ times a constant. We prove an analogous result for arcs between cusps, where we define the length of such an arc to be the length of the sub-arc obtained by removing certain cuspidal regions from the surface.
\end{abstract}

\section{Introduction}

Let $S$ be an orientable surface of negative Euler characteristic of genus $g$ with $n$ boundary components and $p$ punctures, and we assume $(g,n+p)\neq(0,3)$. Let $\Mod(S)$ be the mapping class group
\[\Mod(S)\coloneqq\Mod(S^o)=\text{Homeo}^+(S^o)/\text{Homeo}^+_0(S^o),\]
where $S^o=S\setminus\de S$ is the interior of $S$, $\text{Homeo}^+(S^o)$ is the space of orientation-preserving homeomorphisms of $S^o$ and $\text{Homeo}^+_0(S^o)$ is the subgroup of homeomorphisms properly homotopic to the identity. Let $\PMod(S)$ be the pure mapping class group: the finite-index subgroup of $\Mod(S)$ consisting of exactly those elements which fix each boundary component and each puncture of $S$. See \cite{Primer} for a thorough treatment of mapping class groups. Here we will say that two multicurves, by which we mean formal sums of finitely many weighted curves, are \emph{of the same type} if they share a $\PMod(S)$-orbit.

A celebrated theorem of Mirzakhani (\cite{Mirz08},\cite{Mirz16}) gives the asymptotic growth of the number of (homotopy classes of)  multicurves of the same type of bounded hyperbolic length. Letting $Y$ be a complete hyperbolic metric on $S^o$ and $\gamma_0$ be a multi-curve on $S$, Mirzakhani showed that \begin{equation}
    \lim_{L\to\infty}\frac{|\{\gamma \ \textup{of type} \ \gamma_0\mid\ell_Y(\gamma)\leq L\}|}{L^{6g-6+2(n+p)}}=\frac{\cc(\gamma_0)}{\bb_{g,n+p}}\mthu(\{\ell_Y(\cdot)\leq1\})
    \label{eq:th mirz}
\end{equation}
where $\mthu$ is the Thurston measure on the space $\ML(S)$ of compactly supported measured laminations on $S^o$, $\cc(\gamma_0)$ is a constant depending on the type $\gamma_0$, and $\bb_{g,n+p}$ is a constant depending on the topology of $S$. We refer the reader to \cite{Mirz08} and \cite{Mirz16} for details of the constants, and to \cite{PenHar}, \cite{Thurstonnotes} and \cite{Hatcher88} for a background on $\ML(S)$. Here, $\ell_Y(\gamma)$ denotes the $Y$-length of the geodesic representative of $\gamma$. Mirzakhani first proved the above result for simple multicurves in \cite{Mirz08}, and then again for general multicurves in \cite{Mirz16}; also see \cite{ES19-RA} and \cite{ES20-BOOK} for an alternative proof of this theorem.

In fact, Mirzakhani's theorem holds if we redefine the type of a multicurve to correspond to the orbit of \emph{any} finite-index subgroup of $\Mod(S)$. 

In this paper, we shall show that Mirzakhani's theorem holds when we replace multicurves with multi-arcs. By $\de S$ we shall mean the boundary of $S$, consisting of the $n$ boundary components. The $p$ punctures correspond to ends of $S$. We will count both compact arcs, whose endpoints lie on $\de S$, and infinite arcs, whose endpoints are punctures of $S$. A multi-arc is a formal sum of finitely many weighted arcs, and as with multicurves, two multi-arcs are of the same type if they share a $\PMod(S)$-orbit. Detailed definitions can be found in Section \ref{sec:background}.

We endow $S$ with a (metrically) complete, finite-area, hyperbolic metric $X$ such that $\de S$ is geodesic. We note that whilst our metric is metrically complete, it is not geodesically complete (unless $\de S=\emptyset$), and that the punctures correspond to cusps under such a metric.

The length of a compact arc is the length of its orthogeodesic representative, that is, the unique geodesic compact arc in its homotopy class which meets $\de S$ orthogonally. The study of orthogeodesics on hyperbolic surfaces has a rich history. For example, if one counts all orthogeodesics of length at most $L$, Basmajian's Identity \cite{Bas93} gives an upper bound exponential in $L$ for this number, and the actual asymptotic growth was shown to be exponential by Parkonnen and Paulin in \cite{PP} (see \cite{He18} for a generalisation). These results can be viewed as analogues to Huber's \cite{Huber} and Margulis' \cite{Margulis} Prime Geodesic theorem, showing asymptotic exponential growth of the number of closed geodesics on the surface. Our results are instead analogues of Mirzakhani's theorem, counting arcs in each (pure) mapping class group orbit and giving polynomial asymptotic growth. We prove the following.

\begin{theorem}
    Let $X$ be a complete, finite-area, hyperbolic metric on $S$ such that $\de S\neq\emptyset$ is geodesic. Let $\alpha_0$ be a compact multi-arc on $S$. Then we have
    \[\lim_{L\to\infty}\frac{|\{\alpha \ \textup{of type} \ \alpha_0\mid\ell_X(\alpha)\leq L\}|}{L^{6g-6+2(n+p)}}=\frac{\cc(\alpha_0)}{\bb_{g,n+p}}\mthu(\{\ell_X(\cdot)\leq1\})\]
    where $\mthu$ and $\bb_{g,n+p}$ are as in \eqref{eq:th mirz}, and $\cc(\alpha_0)>0$ is a constant depending only on the type $\alpha_0$.
    \label{th:main compact}
\end{theorem}

As implied by the name, infinite arcs have infinite length as they descend infinitely far down the cusps. Hence we must define a suitable notion of the length of infinite arcs to allow us to derive an analogue of Theorem \ref{th:main compact}. A natural way to do this is to cut off the cusps at some point and consider the length of the segment of the arc which remains. Given any positive $t\leq1$, we define the \emph{$t$-length} of an infinite arc $\alpha$ to be $\ell_X^t(\alpha)=\ell_X(\alpha^t)$, where $\alpha^t$ is the compact sub-arc of $\alpha$ obtained by removing a cuspidal region of area $t$ at each cusp (we refer to Section \ref{sec: infinite arcs} for the precise definition). We prove the following result.

\begin{theorem}
    Let $X$ be a complete, finite-area, hyperbolic metric on $S$ with (possibly empty) geodesic boundary. Let $\alpha_0$ be an infinite arc on $S$. Then for any positive $t\leq1$, we have
    \[\lim_{L\to\infty}\frac{|\{\alpha \ \textup{of type} \ \alpha_0\mid\ell_X^t(\alpha)\leq L\}|}{L^{6g-6+2(n+p)}}=\frac{\cc(\alpha_0)}{\bb_{g,n+p}}\mthu(\{\ell_X(\cdot)\leq1\})\]
    where $\mthu$ and $\bb_{g,n+p}$ are as in \eqref{eq:th mirz}, and $\cc(\alpha_0)>0$ is a constant depending on the type $\alpha_0$. In particular, the limit does not depend on $t$.
    \label{th:main infinite}
\end{theorem}

\begin{remark}
    Theorem \ref{th:main infinite} also holds for infinite multi-arcs, following the same argument presented in this paper. Moreover, our arguments can be easily modified to apply to half-infinite arcs, by which we mean arcs with one endpoint on the boundary and one at a puncture, and multi-arcs whose components are any combination of infinite, compact and half-infinite arcs.
\end{remark}

There are other natural choices of length to assign to infinite arcs, such as the truncated length (see \cite{Parlier20}) and the closely related $\lambda$-length (see \cite{Penner87} and \cite{Pennernotes}). Theorem \ref{th:main infinite} also holds for the truncated length, as we will explain in Section \ref{sec: infinite arcs}.

Although it might be possible to prove Theorems \ref{th:main compact} and \ref{th:main infinite} using methods similar to the proof of Mirzakhani's theorem, we instead take a simpler approach here. The main idea is to associate a multicurve $\gamma_\alpha$ to each multi-arc $\alpha$ in a way which respects length, up to a well-behaved error, and then use Mirzakhani's theorem to deduce Theorems \ref{th:main compact} and \ref{th:main infinite}. In fact, $\cc(\alpha_0)$ will be shown in each case to be closely related to $\cc(\gamma_{\alpha_0})$: we will get that $\cc(\alpha_0)=k(\alpha_0)2^{6g-6+2(n+p)}\cc(\gamma_{\alpha_0})$ where $k(\alpha_0)$ is a combinatorial constant depending on $\alpha_0$, $\cc(\gamma_{\alpha_0})$ is as in \eqref{eq:th mirz}, and $\gamma_{\alpha_0}$ is the curve associated to $\alpha_0$ as defined in Sections \ref{sec:arcs and curves} and \ref{sec: infinite arcs}.

In Section \ref{sec:background}, we will introduce the necessary tools to formulate our proof, and deal with a technicality regarding the application of Mirzakhani's theorem in our setting.  We shall discuss the link between compact arcs and curves in Section \ref{sec:arcs and curves}, before proving Theorem \ref{th:main compact} in Section \ref{sec: compact arcs}. Then in Section \ref{sec: infinite arcs}, we will demonstrate how to apply the same method to infinite arcs and subsequently prove Theorem \ref{th:main infinite}.

\begin{ack}
    The author would like to thank Viveka Erlandsson for suggesting the topic, and for advising them throughout. Additionally, the author would like to thank Hugo Parlier and Juan Souto for their invaluable help and feedback, and Lars Louder for providing Example \ref{ex: lars example}. Finally, the author would like to thank their family and friends for their constant support.
\end{ack}

\section{Background}

\label{sec:background}

As above, let $\de S$ denote the boundary of $S$, and let $S^o=S\setminus\de S$ denote the interior of $S$. Let $\mathfrak{C}$ denote the collection of punctures on $S$. When convenient, we may consider the punctures as marked points on (the closure of) $S$. Throughout the following, let $X$ denote a complete, finite-area, hyperbolic metric on $S$ such that each component of $\de S$ is geodesic.
    
By a \emph{curve} we mean (the homotopy class of) an immersion of the circle $\gamma\colon\mathds{S}^1\to S$, and we identify curves which differ by an orientation. We assume curves to be essential, meaning not homotopic to a point or a puncture, and non-peripheral, meaning not homotopic to a boundary component. By abuse of notation, we will use $\gamma$ to refer to both a curve and its homotopy class. If a curve can be realised by an embedding, we call it simple.
    
A \emph{compact arc} is an immersion of the closed interval $\alpha\colon[0,1]\to S$ such that $\alpha(0),\alpha(1)\in\de S$ and $\alpha((0,1))\subset S^o$. We consider compact arcs up to homotopy relative to $\de S$, where we allow the endpoints to move along $\de S$, and we assume that they are not homotopic into the boundary. Similarly, we define an \emph{infinite arc} to be an immersion of the open interval $\alpha\colon(0,1)\to S$ such that the endpoints are in $\mathfrak{C}$, by which we mean that when we consider the punctures as marked points, the limit of $\alpha$ in each direction is a marked point. We consider infinite arcs up to homotopy relative to $\mathfrak{C}$, and we assume that they are not homotopic into $\mathfrak{C}$. We identify arcs which differ by an orientation, and again by abuse of notation, we refer to both an arc and its homotopy class by $\alpha$. If an (infinite or compact) arc can be realised as an embedding, then we call it simple. We stress that throughout, we allow arcs to have self-intersections; we do not only count simple arcs.
    
A \emph{multicurve} or a \emph{multi-arc} is a finite formal sum of weighted curves or (infinite or compact) arcs respectively. Explicitly, if $\omega$ is a multicurve (resp. multi-arc), then
\[\omega=\sum_{i=1}^m a^i\omega^i\]
for some $a^i\in\mathds{R}_+$ and $m\in\mathds{Z}_+$, where each $\omega^i$ is a curve (resp. arc). We will refer to the $\omega^i$ as the components of $\omega$.

Each homotopy class of curves has a unique geodesic representative, and each homotopy class of compact arcs has a unique geodesic representative which meets the boundary orthogonally, which we refer to as an \emph{orthogeodesic}. We define the length of (a homotopy class of) a curve or compact arc to be the length of its geodesic or orthogeodesic representative, which we denote by $\ell_X(\cdot)$. The length of a multicurve or compact multi-arc is defined to be the weighted sum of the lengths of its components: for $\omega=\sum_{i=1}^m a^i\omega^i$, we have $\ell_X(\omega)=\sum_{i=1}^m a^i\ell_X(\omega^i)$. We will discuss how to assign appropriate finite lengths to infinite arcs in Section \ref{sec: infinite arcs}.

The pure mapping class group $\PMod(S)$ acts naturally on curves and arcs in $S$. If $\varphi$ is a mapping class and $\omega$ is either a geodesic curve, an orthogeodesic compact arc or a geodesic infinite arc, then we define $\varphi\cdot\omega$ to be the (ortho-)geodesic representative of $f(\omega)$, where $f$ is any representative of $\varphi$. Let $\omega_0$ be a curve or arc, then for any curve or arc $\omega$, we say that $\omega$ is \emph{of type} $\omega_0$ if they share an orbit in the pure mapping class group, that is, there exists some $\varphi\in\PMod(S)$ such that $\varphi\cdot\omega_0=\omega$.

The action of $\PMod(S)$ on multicurves and multi-arcs is defined analogously to the above: if $\omega=\sum_{i=1}^m a^i\omega^i$ is a multicurve or multi-arc, then
\[\varphi\cdot\omega=\sum_{i=1}^m a^i(\varphi\cdot\omega^i).\]
We say that a multicurve or multi-arc $\omega$ is of type $\omega_0$ if $\omega$ and $\omega_0$ share a $\PMod(S)$-orbit. As a result, we have that if $\omega=\sum_{i=1}^m a^i\omega^i$ and $\omega_0=\sum_{j=1}^n a^j_0\omega^j_0$ are of the same type then $m=n$ and, up to relabelling, for all $i\in\{1,\dots,m\}$, $a^i=a^i_0$ and $\omega^i$ is of type $\omega^i_0$.

Since $X$ is complete and finite-area, each puncture corresponds to a \emph{cusp} under this metric. Recall that a cusp is an end which has a neighbourhood $H_t$ isometric to
\[\Big\{z\in\mathds{H}^2\Big\vert \textup{Im}(z)>\frac{1}{t}\Big\}/\langle z\mapsto z+1\rangle\]
for some $t>0$, where we have identified the hyperbolic plane $\mathds{H}^2$ with the Poincar\'e upper half-plane. Such a region has volume $t$, and we refer to $H_t$ as a \emph{cuspidal region (of volume $t$)}. The ends of any infinite arc escape down cusps, and the unique geodesic representative of its homotopy class eventually intersects the horocyclic foliation of the corresponding cusps orthogonally.

\newpage

Let $t>0$. For each $p\in\mathfrak{C}$, let $H_t^p$ denote the cuspidal region at $p$ of area $t$. Denote the union of these regions over all $p$ by $\mathscr{H}_t=\cup_{p\in\mathfrak{C}}H_t^p$. It is well-known that for any $t<2$, the cuspidal regions $H_t^p$ are embedded and pairwise disjoint, as can be seen as a result of the Collar Lemma (see for example Theorem 4.4.6 of \cite{Buser}).

For each boundary curve $\delta$ in $\de S$, and for any $c>0$, define the annulus $A_c^\delta$ to be the set of points at a distance less than $c$ from $\delta$. That is,
\[A_c^\delta=\{x\in S\mid d_X(x,\delta)<c\}.\]
Denote by $\mathscr{A}_c=\cup_\delta A_c^\delta$ the union of these annuli over all $\delta$ in $\de S$. It again follows from the Collar Lemma, applied to the boundary curves, that there exists $c'>0$ depending on $X$ such that the annuli $A_{c'}^\delta$ are embedded and pairwise disjoint. We can choose $c'$ such that for all $t<2$, $\mathscr{H}_t\cap\mathscr{A}_{c'}=\emptyset$, and in particular, $S\setminus\overline{(\mathscr{H}_t\cup\mathscr{A}_{c'})}$ is homeomorphic to $S^o$.

For any $p$ and $t<2$, any geodesic segment in $H^p_t$ either never leaves the cuspidal region and so intersects every horocycle in $H^p_t$ orthogonally, or it winds around the cusp before leaving the cuspidal region, and hence, when long enough, creates self-intersections. In the latter case, we call the segment \emph{returning}. In fact, the deeper into $H_t^p$ a returning segment goes the more times it must self-intersect, and there is a direct relationship between the length of a returning segment and its self-intersection number which we record below for future reference. We refer to \cite{Bas13} and \cite{BPT20} for more details about the behaviour of returning segments and for the proof of the below lemma: in particular, Proposition 3.4 of \cite{BPT20} gives a much more precise description of the relationship between how far an arc goes into a cusp and its self-intersection number. Letting $\iota(\cdot,\cdot)$ denote the (geometric) intersection number between curves or arcs (which is realised by their (ortho-)geodesic representatives), we have:

\begin{lemma}
    Let $d>0$, and suppose $\beta$ is a geodesic segment in $\mathscr{H}_1$ with both endpoints on $\de\mathscr{H}_1$ such that $\iota(\beta,\beta)\leq d$. Then there exists some positive $B=B(d)$ such that
    \[\ell_X(\beta)\leq B.\]
    \label{basmajian}
\end{lemma}
\vspace{-\baselineskip}
It follows from Lemma \ref{basmajian} that any geodesic curve $\gamma$ with at most $d$ self-intersections never enters $\mathscr{H}_{\ee^{-B(d)}}$. Moreover, whenever a complete geodesic enters a small annulus around a boundary curve $\delta$, it spirals towards $\delta$, and unless it is asymptotic to $\delta$ it eventually leaves the annulus, creating self-intersections if long enough. It follows that if $\gamma$ is a geodesic curve with at most $d$ self-intersections, there exists some $c<c'$ depending on $d$ (and $X$) such that $\gamma$ never enters $\mathscr{A}_c$. Putting this together with the above gives us that $\gamma$ is contained in the compact subsurface $S\setminus(\mathscr{H}_{\ee^{-B(d)}}\cup\mathscr{A}_c)\subset S^o$. Note that as before, $S\setminus\overline{(\mathscr{H}_{\ee^{-B(d)}}\cup\mathscr{A}_c)}$ is homeomorphic to $S^o$. Furthermore, since $\Mod(S)$ preserves the self-intersection number of curves and arcs, the above is true for any curve of type $\gamma$. We summarise this well-known fact below for reference (and again refer to \cite{Bas13} and \cite{BPT20} for much more precise details).

\begin{lemma}
    Let $\gamma_0$ be a curve. Then there exists a compact subsurface $K\subset S^o$ with $K^o$ homeomorphic to $S^o$ such that for any $\gamma$ of type $\gamma_0$, the geodesic representative of $\gamma$ is contained in $K$.
    \label{lem:compact K}
\end{lemma}

Since multicurves have finitely many components, this lemma holds for multicurves by taking the union of the compact subsurfaces given for (the support of) each component.

Let $d$ be some non-negative integer, and let $\alpha$ be an infinite arc such that $\iota(\alpha,\alpha)=d$. Then similarly to the above, $\alpha\cap\mathscr{H}_{\ee^{-B(d)}}$ consists of exactly 2 components, which are simple half-infinite arcs. Equivalently, $\alpha\cap(S\setminus\mathscr{H}_{\ee^{-B(d)}})$ has exactly one component. We state this here for reference.

\begin{lemma}
    Let $\alpha$ be an infinite arc. Then there exists some positive $t_\alpha<1$, depending only on $\iota(\alpha,\alpha)$, such that $\alpha\cap(S\setminus\mathscr{H}_{t_\alpha})$ has exactly one component.
    \label{lem:arcs give td}
\end{lemma}

We also need the fact that if a geodesic goes far enough into a cusp then it must intersect itself inside $\mathscr{H}_2$. To see this, suppose $\beta$ is a returning geodesic segment in $\mathscr{H}_2$ that enters $H^p_t$ for some $t\leq1$ and some $p\in\mathfrak{C}$. Consider the cuspidal region $H^p_2$ and identify it with
\[\Big\{z\in\mathds{H}^2\Big  \vert \ \textup{Im}(z)>\frac{1}{2}\Big\}/\langle z\mapsto z+1\rangle.\]
A fundamental domain for the action of $z\mapsto z+1$ is the region in $\mathds{H}^2$ bounded by $x=0$ and $x=1$. Note that any geodesic in $\mathds{H}^2$ neither of whose endpoints are at $\infty$ which intersects the line $y=\frac{1}{t}$ also intersects its translate under the map $z\mapsto z+1$, and this intersection occurs above the line $y=\frac{1}{2}$. Hence $\beta$ intersects itself inside the embedded cuspidal region $H^p_2$. In particular, any simple geodesic not asymptotic to a puncture cannot enter $\mathscr{H}_1$. We record this here for reference, and refer to Section 1.3 of \cite{Mcshane} and Proposition 3.2 of \cite{BPT20} for more details.

\begin{lemma}
    Let $0<t\leq1$. If $\beta$ is a geodesic segment in $\mathscr{H}_2$ with both endpoints on $\de\mathscr{H}_2$, and $\beta\cap\mathscr{H}_t\neq\emptyset$, then $\iota(\beta,\beta)\geq1$.
    \label{lem:beta self-intersects at least once}
\end{lemma}

Recall that $\mathcal{ML}(S)$ is the space of compactly supported measured laminations on $S^o$. The support of any $\lambda\in\mathcal{ML}(S)$ is a union of simple geodesics and thus it follows from Lemma \ref{lem:beta self-intersects at least once} that the support of $\lambda$ is contained in $S^o\setminus\mathscr{H}_1$. In fact, there exists a compact subsurface $K\subseteq S^o\setminus\mathscr{H}_1$ which contains the support of $\mathcal{ML}(S)$.

As mentioned in the introduction, the idea in this paper is to find a nice way to associate curves to arcs so that we can use Mirzakhani's theorem about counting curves to count arcs. Mirzakhani's theorem is stated for complete finite-area hyperbolic metrics on the interior $S^o$, and we will need to use the result for our metric $X$ on $S$ which has geodesic boundary. This issue is resolved by instead using the following generalisation.
\begin{theorem}[\cite{EPS16}, Corollary 1.3]
    Let $Y$ be a complete Riemannian metric on $S^o=S\setminus\de S$. Then for any multicurve $\gamma_0$,
    \[\lim_{L\to\infty}\frac{|\{\gamma \ \textup{of type} \ \gamma_0\mid\ell_Y(\gamma)\leq L\}|}{L^{6g-6+2(n+p)}}=\frac{\cc(\gamma_0)}{\bb_{g,n+p}}\mthu(\{\ell_Y(\cdot)\leq1\})\]
    where $\mthu$, $\cc(\gamma_0)$ and $\bb_{g,n+p}$ are as in \eqref{eq:th mirz} and $\ell_Y(\gamma)$ is the length of a shortest curve homotopic to $\gamma$.
    \label{th:EPS counting}
\end{theorem}
To see that this result implies that we can count curves in our setting, let $\gamma_0$ be a multicurve on $S$ and let $K=K(\gamma_0)$ be the compact subsurface of $S^o$ given by Lemma \ref{lem:compact K}. By the discussion after Lemma \ref{lem:beta self-intersects at least once}, we may assume that $K$ is such that $\ML(S)\subset K$. Take any complete Riemannian metric $Y$ on $S^o$ which agrees with $X$ on $K$. Since the geodesic representative of every multicurve $\gamma$ of type $\gamma_0$ is contained in $K$ and $\ML(S)\subset K$, we have that $\ell_Y(\gamma)=\ell_X(\gamma)$ for all $\gamma$ of type $\gamma_0$ and $\ell_Y(\lambda)=\ell_X(\lambda)$ for all $\lambda\in\ML(S)$. We record this consequence for reference.
\begin{corollary}
    Let $X$ be a complete, finite-area, hyperbolic metric on $S$ such that $\de S$ is geodesic. Then for any multicurve $\gamma_0$,
    \[\lim_{L\to\infty}\frac{|\{\gamma \ \textup{of type} \ \gamma_0\mid\ell_X(\gamma)\leq L\}|}{L^{6g-6+2(n+p)}}=\frac{\cc(\gamma_0)}{\bb_{g,n+p}}\mthu(\{\ell_X(\cdot)\leq1\})\]
    where $\mthu$, $\cc(\gamma_0)$ and $\bb_{g,n+p}$ are as in \eqref{eq:th mirz}.
    \label{corol:count with X}
\end{corollary}

\section{Compact arcs}

\subsection{Relating compact arcs and curves}

\label{sec:arcs and curves}

In this section, we will discuss how to associate multicurves to compact multi-arcs in a way that respects length, up to some well-behaved error. We will achieve this using the nice geometric properties of pairs of pants.

First, we will discuss how to associate a single curve to a single compact arc. Fix an orientation on $S$, which induces an orientation on the boundary components. Let $\alpha$ be some compact arc in $S$, oriented from $\alpha(0)$ to $\alpha(1)$. The endpoints $\alpha(0)$ and $\alpha(1)$ each lie on a boundary component, which we denote by $\delta_0^\alpha$ and $\delta_1^\alpha$ respectively: note that these are not necessarily distinct. Pick basepoints $p_0$ and $p_1$ on $\delta_0^\alpha$ and $\delta_1^\alpha$ respectively, and consider these boundary components as loops based at their respective basepoints. Apply a homotopy to $\alpha$ so that $\alpha(0)=p_0$ and $\alpha(1)=p_1$. Then we define the \emph{curve associated to} $\alpha$ to be the geodesic curve $\gamma_\alpha$ (freely) homotopic to the concatenated path
\begin{equation*}
    \alpha^{-1}\cdot\delta_1^\alpha\cdot\alpha\cdot\delta_0^\alpha
\end{equation*}
which starts and ends at $p_0$. In particular, in the case that $\alpha$ is simple and $\delta_0^\alpha\neq\delta_1^\alpha$, $\gamma_\alpha$ is homotopic to the boundary of a small neighbourhood of the union of $\alpha$, $\delta_0^\alpha$ and $\delta_1^\alpha$. See Figure \ref{fig:arcs and curves}. Recall that we identify arcs and curves that differ by an orientation, and note that the arc $\alpha'$ which differs from $\alpha$ only in orientation gives rise to exactly the same curve as $\alpha$, even in orientation. 

\begin{figure}[h]
    \centering
    \includegraphics[scale=.34]{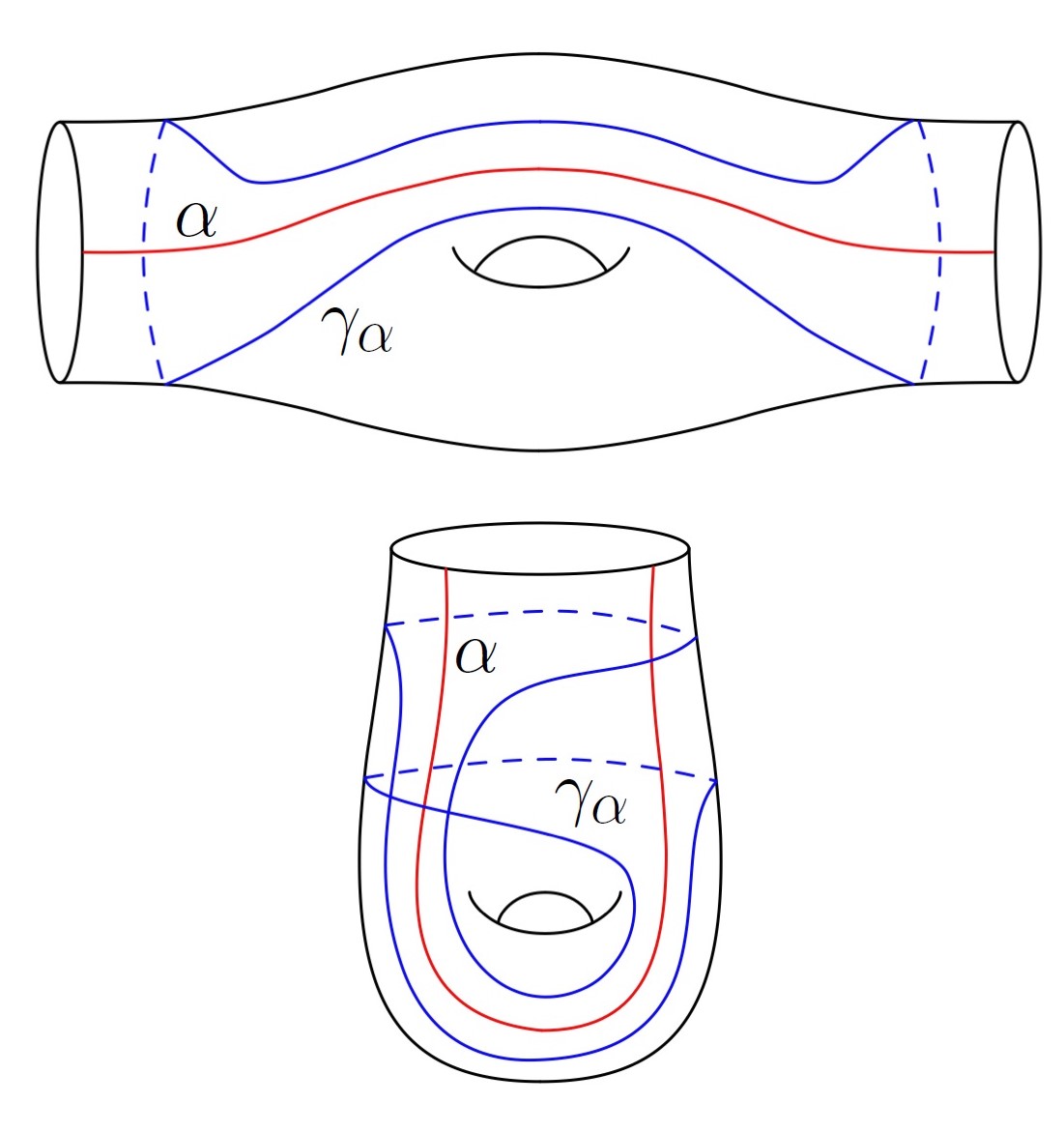}
    \caption{Examples of compact arcs (in red) and their associated curves (in blue).}
    \label{fig:arcs and curves}
\end{figure}

Now, let $P$ be a topological surface with $g=0$, $n=3$ and $p=0$ known as a pair of pants, and fix an orientation on $P$. The boundary components of $P$ are referred to as \emph{cuffs}, and for each pair of cuffs the unique homotopy class of simple compact arcs between them is called a \emph{seam}. We label the cuffs by $\delta_0^P, \delta_1^P$ and $\delta_2^P$ and the seam between $\delta_0^P$ and $\delta_1^P$ by $\alpha_P$.

For any arc $\alpha$ on $S$ with endpoints on $\delta_0^\alpha$ and $\delta_1^\alpha$, there exists an orientation-preserving immersion $\ia\colon P\to S$ such that
\begin{align*}
    \ia(\delta_0^P)&=\delta_0^\alpha, \\
    \ia(\delta_1^P)&=\delta_1^\alpha, \\
    \ia(\alpha_P)&=\alpha.
\end{align*}
Note that the images of the two cuffs and the seam under this map determine the image of the third cuff up to homotopy, since this is exactly the (free) homotopy class of $\alpha^{-1}\cdot\delta_1^\alpha\cdot\alpha\cdot\delta_0^\alpha$. That is,
\begin{equation}
    \gamma_\alpha=\iota_\alpha(\delta_2^P)
    \label{gamma alpha is iota delta}
\end{equation}
(up to homotopy). Let $\ia'$ be another immersion which satisfies the above. Then since they agree on $\delta_0^P$, $\delta_1^P$ and $\alpha_P$, we have that the images $\ia(\delta_2^P)$ and $\ia'(\delta_2^P)$ of the third cuff are homotopic and such a homotopy extends to a homotopy from $\ia(P)$ to $\ia'(P)$. Thus any two such immersions of $P$ are homotopic. 

In the case that $\alpha=\sum_{i=1}^m a^i\alpha^i$ is a compact multi-arc, we define the \emph{multicurve associated to} $\alpha$ to be the weighted sum of the the curves associated to its components. That is,
\begin{equation}
    \gamma_\alpha=\sum_{i=1}^m a^i\gamma_{\alpha^i}=\sum_{i=1}^m a^i\iota_{\alpha^i}(\delta_2^P).
    \label{def:ga multi}
\end{equation}
In the remainder of the section, we will first prove several statements for single compact arcs before demonstrating how these also hold for compact multi-arcs.

Let $\mathcal{A}(S)$ and $\mathcal{C}(S)$ denote the sets of compact arcs and curves on $S$ respectively. We define the \emph{association map} $I\colon\mathcal{A}(S)\to\mathcal{C}(S)$ by
\[I(\alpha)=\gamma_\alpha.\]
First, we show that $I$ distorts the length of arcs in a controlled way.

\begin{lemma}
    Let $X$ be a complete, finite-area, hyperbolic metric on $S$ such that $\de S$ is geodesic. There exists a constant $C(X)>0$ such that for any $\alpha\in\mathcal{A}(S)$, 
    \[|\ell_X(I(\alpha))-2\ell_X(\alpha)|\leq C(X),\]
    where $I(\alpha)=\gamma_\alpha$ is the curve associated to $\alpha$.
    \label{lem:bound compact}
\end{lemma}
\begin{proof}
    This will follow from basic hyperbolic geometry. Let $\alpha\in\mathcal{A}(S)$. 
    
    Choosing the lengths of 2 cuffs and the seam between them on a pair of pants fixes the length of the third cuff; that is, the lengths of $\delta_0^\alpha$, $\delta_1^\alpha$ and $\alpha$ determine the length of $\gamma_\alpha$. More precisely,
    \begin{equation}
        \cosh{\frac{\ell_X(\gamma_\alpha)}{2}}=\sinh{\frac{\ell_X(\delta_0^\alpha)}{2}}\sinh{\frac{\ell_X(\delta_1^\alpha)}{2}}\cosh{\ell_X(\alpha)}
        -\cosh{\frac{\ell_X(\delta_0^\alpha)}{2}}\cosh{\frac{\ell_X(\delta_1^\alpha)}{2}}
        \label{eq:cosh eq}
    \end{equation}
    (See Theorem 2.4.1 of \cite{Buser}). Let $\de S=\{\delta_1,\dots,\delta_n\}$, then for some $i,j$, $\delta_0^{\alpha}=\delta_i$ and $\delta_1^{\alpha}=\delta_j$. To simplify notation, we will write
    \begin{align*}
        A_{i,j}&=\sinh{\frac{\ell_X(\delta_i)}{2}}\sinh{\frac{\ell_X(\delta_j)}{2}}, \\
        B_{i,j}&=\cosh{\frac{\ell_X(\delta_i)}{2}}\cosh{\frac{\ell_X(\delta_j)}{2}}.
    \end{align*}
    Then equation \eqref{eq:cosh eq} gives the length of $\gamma_\alpha$ as
    \[\ell_X(\gamma_\alpha)=2\cosh^{-1}(A_{i,j}\cosh{\ell_X(\alpha)}-B_{i,j}),\]
    and we want to show that this length is close to $2\ell_X(\alpha)$. To this end, we define the error function $E_{i,j}\colon [m_{i,j},\infty)\to\mathds{R}$ by
    \[E_{i,j}(\ell)=2\cosh^{-1}(A_{i,j}\cosh{\ell}-B_{i,j})-2\ell,\]
    where $m_{i,j}$ is a lower bound on the lengths of arcs between $\delta_i$ and $\delta_j$ which can be taken as \\ ${m_{i,j}\coloneqq\cosh^{-1}\Big(\frac{B_{i,j}+1}{A_{i,j}}\Big)}.$
    This function is continuous, and the limit
    \[\lim_{\ell\to\infty}E_{i,j}(\ell)=2\ln(A_{i,j})\]
    exists. Hence $|E_{i,j}(\ell)|$ is bounded for all $\ell\in[m_{i,j},\infty)$, thus there exists $C(i,j)>0$ such that for any arc $\alpha$ between $\delta_i$ and $\delta_j$, we have $|\ell_X(\gamma_\alpha)-2\ell_X(\alpha)|\leq C(i,j)$. Therefore as S has finitely many boundary components, there exists $C(X)>0$ such that for any $\alpha\in\mathcal{A}(S)$,    \[|\ell_X(\gamma_\alpha)-2\ell_X(\alpha)|\leq C(X).\]
\end{proof}

This association map $I$ is not one-to-one: for example, suppose $S$ is a four-holed sphere with boundary components $\delta_1,\delta_2,\delta_3,\delta_4$. Let $\alpha$ be a simple arc connecting $\delta_1$ and $\delta_2$, and $\beta$ a simple arc connecting $\delta_3$ and $\delta_4$. Then $\gamma_\alpha$ and $\gamma_\beta$ are the same (homotopy class of) curve, up to orientation.

There are also less trivial examples, such as where $\alpha$ and $\beta$ are arcs between the same boundary components, as the following example illustrates.

\begin{example}
    Let $P$ be the pair of pants from above. Fix a base point $\star\in P$ and choose generators $A$ and $B$ for $\pi_1(P,\star)\simeq F_2$ such that the loops $A$ and $B$ are freely homotopic (as oriented curves) to $\delta_0^P$ and $\delta_1^P$ respectively. Note that the third boundary component $\delta_2^P$ corresponds to the conjugacy class of $B^{-1}A^{-1}$. Let $S=S_{0,3,0}$ be another oriented pair of pants and let $a$, $b$ be generators of $\pi_1(S,\ast)$ (for some basepoint $\ast\in S$) such that $a$ and $b$ are freely homotopic to two of the boundary components of $S$, say $\delta_0^S$ and $\delta_1^S$ respectively.
    
    Now consider the homomorphism $h_1\colon\pi_1(P,\star)\to\pi_1(S,\ast)$ defined by
    \[A\mapsto a \qquad \textup{and} \qquad B\mapsto b^{-1}aba^{-1}b.\]
    As $h_1(A)$ and $h_1(B)$ are conjugate to $a$ and $b$, $h_1(A)$ and $h_1(B)$ are (freely) homotopic to $\delta_0^S$ and $\delta_1^S$, preserving orientation. Thus $h_1$ induces an immersion $\iota_1\colon P\to S$.
    
    Similarly, the homomorphism $h_2\colon\pi_1(P,\star)\to\pi_1(S,\ast)$ defined by 
    \[A\mapsto ab^{-1}aba^{-1} \qquad \textup{and} \qquad B\mapsto b\]
    also induces an immersion $\iota_2\colon P\to S$.
    
    Note that (the homotopy classes of) all three boundary components of $\iota_1(P)$ and $\iota_2(P)$ agree; they correspond to the conjugacy classes of $a$, $b$ and $b^{-1}ab^{-1}a^{-1}ba^{-1}$. Further note that
    \[\langle a,b^{-1}aba^{-1}b\rangle\neq\langle ab^{-1}aba^{-1}, b\rangle,\]
    and hence $\iota_1(P)$ is not homotopic to $\iota_2(P)$. In particular, the arcs $\iota_1(\alpha_P)$ and $\iota_2(\alpha_P)$ are not homotopic whilst $\gamma_{\iota_1(\alpha_P)}$ and $\gamma_{\iota_2(\alpha_P)}$ determine the same curve (namely, the conjugacy class of $b^{-1}ab^{-1}a^{-1}ba^{-1}$).
    \label{ex: lars example}
\end{example}

\medskip

Despite this complication, if we restrict $I$ to a (pure) mapping class group orbit of an arc, it is uniformly bounded-to-one. This will be enough for our purposes, as $I$ is equivariant with respect to $\PMod(S)$, which we now demonstrate.

\begin{lemma}
    Let $\varphi\in\PMod(S)$ and $\alpha\in\mathcal{A}(S)$. Then
    \[I(\varphi\cdot\alpha)=\varphi\cdot I(\alpha).\]
    \label{lem:I equivariant}
\end{lemma}
\begin{proof}
    To see this, write
    \[\gamma_{\varphi\cdot\alpha}=\iota_{\varphi\cdot\alpha}(\delta_2^P)\]
    using \eqref{gamma alpha is iota delta}. We have that $\iota_{\varphi\cdot\alpha}(P)\subset S$ is an immersed pair of pants with boundary components $\delta_0^{\alpha}$ and $\delta_1^{\alpha}$, and the seam between them is $\varphi\cdot\alpha$. Similarly, $\varphi\cdot\iota_{\alpha}(P)\subset S$ is an immersed pair of pants with boundary components $\delta_0^{\alpha}$ and $\delta_1^{\alpha}$ and seam between them $\varphi\cdot\alpha$, since $\varphi$ fixes the boundary components of $S$. Therefore, $\iota_{\varphi\cdot\alpha}(P)=\varphi\cdot\iota_{\alpha}(P)$, and in particular, $\iota_{\varphi\cdot\alpha}(\delta_2^P)=\varphi\cdot\iota_{\alpha}(\delta_2^P)$. Since $\gamma_{\alpha}=\iota_{\alpha}(\delta^P_2)$ by \eqref{gamma alpha is iota delta}, we have
    \begin{equation*}
    \gamma_{\varphi\cdot\alpha}=\varphi\cdot\gamma_{\alpha}
    \end{equation*} 
    and so the lemma holds.
\end{proof}

Let $\alpha_0$ be a compact arc, and let \[I_{\alpha_0}\colon\PMod(S)\cdot\alpha_0\to\PMod(S)\cdot\gamma_{\alpha_0}\]
be the restriction of $I$ to arcs of type $\alpha_0$. By Lemma \ref{lem:I equivariant}, this map is well-defined. Moreover, as mentioned above, it is uniformly bounded-to-1, which we prove in the following proposition.
\begin{prop}
    Let $\alpha_0$ be a compact arc. Then there exits $k=k(\alpha_0)$ such that $I_{\alpha_0}$ is surjective and $k$-to-1.
    \label{prop:k to 1}
\end{prop}
\begin{proof}
    Let $\gamma\in\PMod(S)\cdot\gamma_{\alpha_0}$, so $\gamma=\varphi\cdot\gamma_{\alpha_0}$ for some $\varphi\in\PMod(S)$. Let $\alpha=\varphi\cdot\alpha_0$. Then by Lemma \ref{lem:I equivariant},
    \[\gamma=\varphi\cdot\gamma_{\alpha_0}=\gamma_{\varphi\cdot\alpha_0}=I_{\alpha_0}(\varphi\cdot\alpha_0)=I_{\alpha_0}(\alpha),\]
    thus $I_{\alpha_0}$ is surjective. 
    
    Consider the collection of compact arcs $\alpha_i$ such that $\gamma_{\alpha_i}=\gamma$. This set is finite: by Lemma \ref{lem:bound compact}, the maximum length of such an arc is $\frac{1}{2}\ell_X(\gamma)+\frac{1}{2}C(X)$, and thus there are only finitely many. Suppose that there are exactly $k$ such arcs, and suppose further that for some other curve $\gamma'$ of type $\gamma_{\alpha_0}$, there are exactly $k'$ arcs $\alpha'_i$ such that $\gamma_{\alpha'_i}=\gamma'$. Since $\gamma$ and $\gamma'$ are of the same type, there exists some $\psi\in\PMod(S)$ such that $\gamma'=\psi\cdot\gamma$. Again by Lemma \ref{lem:I equivariant}, we have that for each $i\in\{1,\dots,k\}$,
    \[\gamma_{\psi\cdot\alpha_i}=\psi\cdot\gamma_{\alpha_i}=\psi\cdot\gamma=\gamma'.\]
    Thus we have constructed a set of $k$ arcs which are associated to $\gamma'$, and so $k\leq k'$. Analogously, we can construct a set of $k'$ arcs which are associated to $\gamma$, hence $k'\leq k$. Therefore $k=k'$, and so $k$ is uniform across all curves of type $\gamma_{\alpha_0}$.
\end{proof}

We will need to use these results for compact multi-arcs. Let $\mathcal{A}_{\text{multi}}(S)$ and $\mathcal{C}_{\text{multi}}(S)$ be the sets of weighted compact multi-arcs and weighted multicurves respectively. By abuse of notation, define the association map on multi-arcs $I\colon\mathcal{A}_{\text{multi}}(S)\to\mathcal{C}_{\text{multi}}(S)$ by
\[I(\alpha)=I\Big(\sum_{i=1}^m a^i\alpha^i\Big)=\sum_{i=1}^m a^i I(\alpha^i).\]
As multi-arcs have finitely many components, the following is a short corollary of Lemma \ref{lem:bound compact}.

\newpage

\begin{corollary}
     Let $X$ be a complete, finite-area, hyperbolic metric on $S$ such that $\de S$ is geodesic. There exists a constant $C(X)>0$ such that for any $\alpha\in\mathcal{A}_{\text{multi}}(S)$, 
    \[|\ell_X(I(\alpha))-2\ell_X(\alpha)|\leq C(X),\]
    where $I(\alpha)=\gamma_\alpha$ is the multicurve associated to $\alpha$.
    \label{cor:bound compact multi}
\end{corollary}
Furthermore, as an immediate corollary to Lemma \ref{lem:I equivariant}, $I$ remains $\PMod(S)$-equivariant when defined on compact multi-arcs. Following the proof of Proposition \ref{prop:k to 1}, we can see that the restriction of the association map to multi-arcs of a particular type is surjective and $k$-to-1, for some $k$ depending only on the type. We record this here for reference.
\begin{corollary}
     Let $\alpha_0$ be a compact multi-arc and $\gamma_{\alpha_0}$ be as in \eqref{def:ga multi}. Let $I_{\alpha_0}\colon\PMod(S)\cdot\alpha_0\to\PMod(S)\cdot\gamma_{\alpha_0}$ be the restriction of $I$ to multi-arcs of type $\alpha_0$. Then there exists $k=k(\alpha_0)$ such that $I_{\alpha_0}$ is surjective and $k$-to-1.
     \label{corol:k-to-1 multi}
\end{corollary}

\begin{remark}
    It would be interesting to understand the value of $k$. For example, if $\alpha_0$ is simple, is $k(\alpha_0)=1$? Is it always 1? Future work could study this association more closely to give us an idea of how it varies with the type of arc, and if it is not identically 1, derive some explicit examples of arcs of the same type which are associated to the same curve.
\end{remark}

\subsection{Counting compact arcs}

\label{sec: compact arcs}

We can now prove Theorem \ref{th:main compact}.
\begin{proof}[Proof of Theorem $\ref{th:main compact}$]
    Let $\alpha_0$ be a compact multi-arc. Consider the set
    \[\{\alpha \ \text{of type} \ \alpha_0\mid\ell_X(\alpha)\leq L\}\]
    for some $L>0$. By Corollary \ref{corol:k-to-1 multi}, there exists $k$ such that the association map $I_{\alpha_0}$ is $k$-to-1, and by Corollary \ref{cor:bound compact multi}, the maximum length of a multicurve associated to a multi-arc in this set is $2L+C(X)$. Hence we can write
    \[|\{\alpha \ \text{of type} \ \alpha_0\mid\ell_X(\alpha)\leq L\}|\leq k|\{\gamma \ \text{of type} \ \gamma_{\alpha_0}\mid\ell_X(\gamma)\leq2L+C(X)\}|.\]
    Then we have
    \begingroup
    \addtolength{\jot}{1em}
    \begin{align*}
        \limsup_{L\to\infty}&\frac{|\{\alpha \ \text{of type} \ \alpha_0\mid\ell_X(\alpha)\leq L\}|}{L^{6g-6+2(n+p)}}\leq\limsup_{L\to\infty}\frac{k|\{\gamma \ \text{of type} \ \gamma_{\alpha_0}\mid\ell_X(\gamma)\leq2L+C(X)\}|}{L^{6g-6+2(n+p)}} \\
        &=k\cdot\limsup_{L\to\infty}\frac{|\{\gamma \ \text{of type} \ \gamma_{\alpha_0}\mid\ell_X(\gamma)\leq2L+C(X)\}|}{(2L+C(X))^{6g-6+2(n+p)}}\frac{(2L+C(X))^{6g-6+2(n+p)}}{L^{6g-6+2(n+p)}} \\
        &=k\cdot2^{6g-6+2(n+p)}\frac{\cc(\gamma_{\alpha_0})}{\bb_{g,n+p}}\mthu(\{\ell_X(\cdot)\leq1\})
    \end{align*}
    \endgroup
    using Corollary \ref{corol:count with X}. Using a similar argument, we have
    \[|\{\alpha \ \text{of type} \ \alpha_0\mid\ell_X(\alpha)\leq L\}|\geq k|\{\gamma \ \text{of type} \ \gamma_{\alpha_0}\mid\ell_X(\gamma)\leq2L-C(X)\}|\]
    and therefore
     \[\liminf_{L\to\infty}\frac{|\{\alpha \ \text{of type} \ \alpha_0\mid\ell_X(\alpha)\leq L\}|}{L^{6g-6+2(n+p)}}\geq k\cdot2^{6g-6+2(n+p)}\frac{\cc(\gamma_{\alpha_0})}{\bb_{g,n+p}}\mthu(\{\ell_X(\cdot)\leq1\}).\]
    Hence, since the limit superior and inferior both exist and agree, we have that the limit exists and equals the same value. In other words,
    \[\lim_{L\to\infty}\frac{|\{\alpha \ \text{of type} \ \alpha_0\mid\ell_X(\alpha)\leq L\}|}{L^{6g-6+2(n+p)}}   =\frac{\cc(\alpha_0)}{\bb_{g,n+p}}\mthu(\{\ell_X(\cdot)\leq1\}),\]
    where $\cc(\alpha_0)\coloneqq k\cdot2^{6g-6+2(n+p)}\cc(\gamma_{\alpha_0})$, $k$ is as in Corollary \ref{corol:k-to-1 multi}, and $\cc(\gamma_{\alpha_0})$, $\bb_{g,n+p}$ and $\mthu$ are as in \eqref{eq:th mirz}.
\end{proof}

\section{Counting infinite arcs}

\label{sec: infinite arcs}

The main work in this section is to prove the lemmas from Section \ref{sec:arcs and curves} for infinite arcs, with modifications to account for the range of values the $t$-length of an infinite arc can take. The proof of Theorem \ref{th:main infinite} will then be analogous to that of Theorem \ref{th:main compact}.

First, we discuss the assignment of appropriate finite lengths to infinite arcs. For any $t\in(0,1]$, let $\mathscr{H}_t=\cup_{p\in\mathfrak{C}}H_t^p$ be the union of the cuspidal regions of volume $t$ as before, and define the $t$\emph{-length} $\ell^t_X(\alpha)$ of any arc $\alpha$ to be the length of $\alpha^t=\alpha\cap(S\setminus\mathscr{H}_t)$. That is,
\[\ell^t_X(\alpha)=\ell_X(\alpha^t).\]
Note that in general, $\alpha^t$ could consist of multiple connected components and in this case, $\ell_X(\alpha^t)$ is the sum of the lengths of its components. Fix an infinite arc $\alpha$ and let $t_\alpha$ be given by Lemma \ref{lem:arcs give td}. Then $\alpha\cap(S\setminus\mathscr{H}_{t_\alpha})$ is connected, and moreover for any $t\leq t_\alpha$, $\alpha^t$ has exactly one component.

As mentioned in the introduction, the $t$-length of an arc is closely related to the \emph{truncated length} as defined by Parlier in \cite{Parlier20}. Choose a standard collection of cuspidal regions, which we may take as $\mathscr{H}_1$. For any infinite arc $\alpha$, the (doubly) truncated length of $\alpha$ is the length of the segment of $\alpha$ between the first and last times $\alpha$ crosses $\de\mathscr{H}_1$. We denote this length by $\ell_X^{Tr}(\alpha)$.

This length is also closely related to the \emph{$\lambda$-length} introduced by Penner in \cite{Penner87} and \cite{Pennernotes}. In our setting, and choosing the appropriate cuspidal regions, we have that
\[\lambda(\alpha)=\ee^{\frac{1}{2}\ell_X^{Tr}(\alpha)}.\]

Note that for any $\alpha\in\mathcal{A}(S)$, $\ell_X^{t_\alpha}(\alpha)$ and the truncated length $\ell_X^{Tr}(\alpha)$ differ by a constant, and this constant depends only on $\iota(\alpha,\alpha)$. This is because the $t_\alpha$-length of $\alpha$ is exactly the truncated length plus the lengths of the two geodesic segments of $\alpha$ between $\de\mathscr{H}_1$ and $\de\mathscr{H}_{t_\alpha}$, which each have length $\ln(\frac{1}{t_\alpha})$. Thus we can write
\begin{equation}
    |\ell_X^{t_\alpha}(\alpha)-\ell_X^{Tr}(\alpha)|\leq2\ln\Big(\frac{1}{t_\alpha}\Big).
    \label{eq: truncated length bound}
\end{equation}
Recall that by Lemma \ref{lem:arcs give td}, $t_\alpha$ depends only on $\iota(\alpha,\alpha)$. Using this fact, one can show that Theorem \ref{th:main infinite} also holds when we replace $\ell_X^t$ by $\ell_X^{Tr}$.

The curve $\gamma_\alpha$ associated to an infinite arc $\alpha$ is defined analogously to the compact case. Denote by $p^\alpha_0$ and $p^\alpha_1$ the cusps at each end of $\alpha$, where $\alpha$ is oriented from $p^\alpha_0$ to $p^\alpha_1$. With $t_\alpha$ as above, define $\gamma_\alpha$ to be the geodesic curve (freely) homotopic to the loop given by the concatenation 
\[(\alpha^{t_\alpha})^{-1}\cdot h^\alpha_1\cdot\alpha^{t_\alpha}\cdot h^\alpha_0,\]
where $h^\alpha_0=\de H_{t_\alpha}^{p_0^\alpha}, h^\alpha_1=\de H_{t_\alpha}^{p_1^\alpha}$ are the horocycles at $p^\alpha_0$ and $p^\alpha_1$ of length $t_\alpha$, viewed as loops with appropriate basepoints and orientations. Note that if we replaced $t_\alpha$ with any $t<t_\alpha$, we would get the same curve $\gamma_\alpha$. Let $P$ be a (generalised) pair of pants with one boundary component and two cusps, labelled $\delta$, $p_0$ and $p_1$ respectively. There is an orientation-preserving immersion $\iota_\alpha\colon P\to S$ which sends $p_0$ and $p_1$ to $p^\alpha_0$ and $p^\alpha_1$ respectively, and such that (the homotopy class of) the simple infinite arc between them is mapped to $\alpha$. Then equivalently, $\gamma_\alpha$ is the geodesic representative of $\iota_\alpha(\delta)$.

We define $J\colon\mathcal{A}_\infty(S)\to\mathcal{C}(S)$ to be the association map from infinite arcs to curves, where $\mathcal{A}_\infty(S)$ is the set of all infinite arcs on $S$. That is, for any $\alpha\in\mathcal{A}_\infty(S)$,
\[J(\alpha)=\gamma_\alpha.\]

We will now prove an analogue of Lemma \ref{lem:bound compact} for infinite arcs. As $t$ can be taken arbitrarily close to 0, the $t$-length of an arc can be arbitrarily long, and so any bound on the difference between the $t$-lengths of infinite arcs and the lengths of their associated curves must depend on $t$. Furthermore, arcs which self-intersect arbitrarily often will go arbitrarily deep into the cusps, and therefore so will their curves. Thus for a fixed value of $t$, this difference can become arbitrarily large. Hence, any such bound must also depend on self-intersection number.

\begin{lemma}
    Let $\alpha$ be an infinite arc. Then for any positive $t<1$, there exists $C\big(\iota(\alpha,\alpha),t\big)>0$ such that
    \[|\ell_X(J(\alpha))-2\ell_X({\alpha}^t)|\leq C\big(\iota(\alpha,\alpha),t\big)\]
    where $J(\alpha)=\gamma_{\alpha}$ is the curve associated to $\alpha$.
    \label{lem:bound infinite}
\end{lemma}

\begin{figure}[h]
    \centering
    \includegraphics[scale=.34, angle=0]{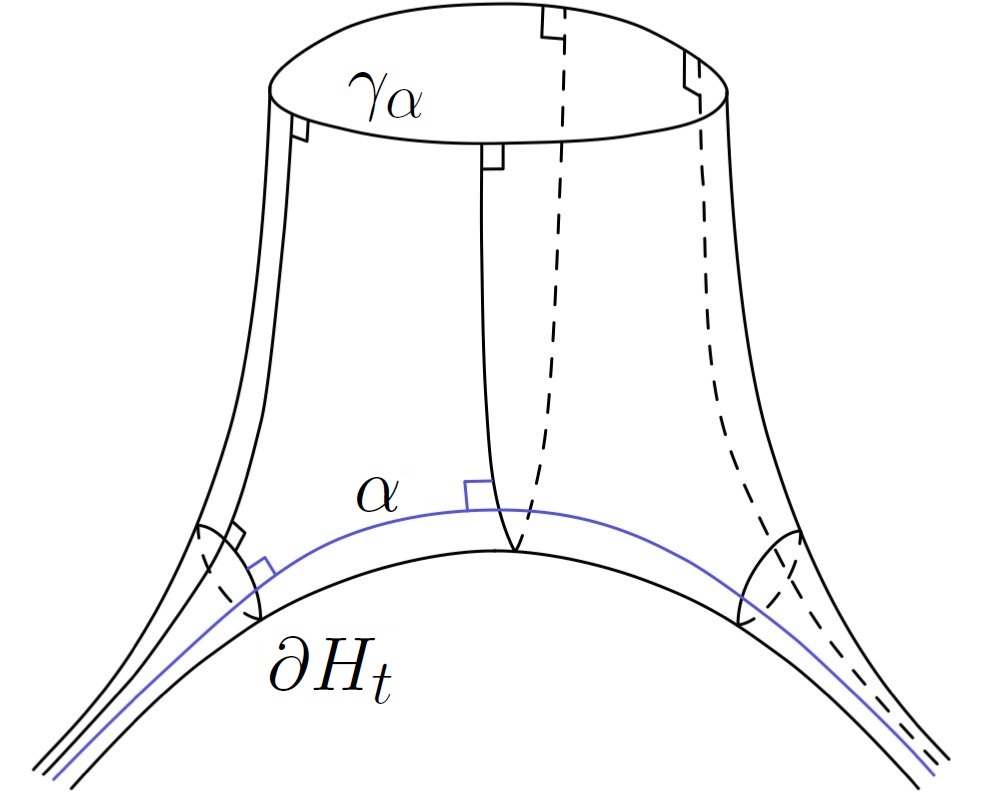}
    \caption{The pre-image of an infinite arc $\alpha$ in the generalised pair of pants $P$, with the perpendiculars which we cut along.}
    \label{fig:infinite arcs pair of pants}
\end{figure}

\begin{proof}
    Let $\alpha$ be an infinite arc, and let $t_\alpha$ be given by Lemma \ref{lem:arcs give td}. We will start by proving the lemma in the case that $t\leq t_\alpha$. Then we will demonstrate that for $t>t_{\alpha}$, the difference between $\ell_X^t(\alpha)$ and $\ell_X^{t_\alpha}(\alpha)$ is uniformly bounded across all arcs with the same self-intersection number, and so complete the proof.
    
    Suppose that $t\leq t_\alpha$. Equip the generalised pair of pants $P$ with a metric using the pullback of $X$ through $\iota_\alpha$. Cut $P$ along four geodesic arcs: the pre-image of $\alpha$, the perpendicular compact simple geodesic arc from the boundary component to itself, and the two simple half-infinite geodesic arcs between the boundary component and the cusps. See Figure \ref{fig:infinite arcs pair of pants}. We are left with 4 isometric copies of a quadrilateral with three right angles and one ideal vertex, which we label as in Figure \ref{fig:quadrilaterals}. 
    
    Since $t\leq t_\alpha$, we have that the length of the edge $qw$ is $\frac{1}{2}\ell_X(\alpha^t)$, and the length of the edge $uv$ is $\frac{1}{4}\ell_X(\gamma_\alpha)$. Consider this quadrilateral in the upper-half space model for $\mathds{H}^2$ and normalise it such that the ideal vertex is at $\infty$ and the edges incident to it are on the lines $x=0$ and $x=1$. As the length of the boundary of the cuspidal region of area $t$ is $t$, the length of the segment which lives in this quadrilateral is $\frac{t}{2}$. Therefore it lies on the line $y=\frac{2}{t}$.
    
    \begin{figure}[h]
        \centering
        \includegraphics[scale=.32, angle=0]{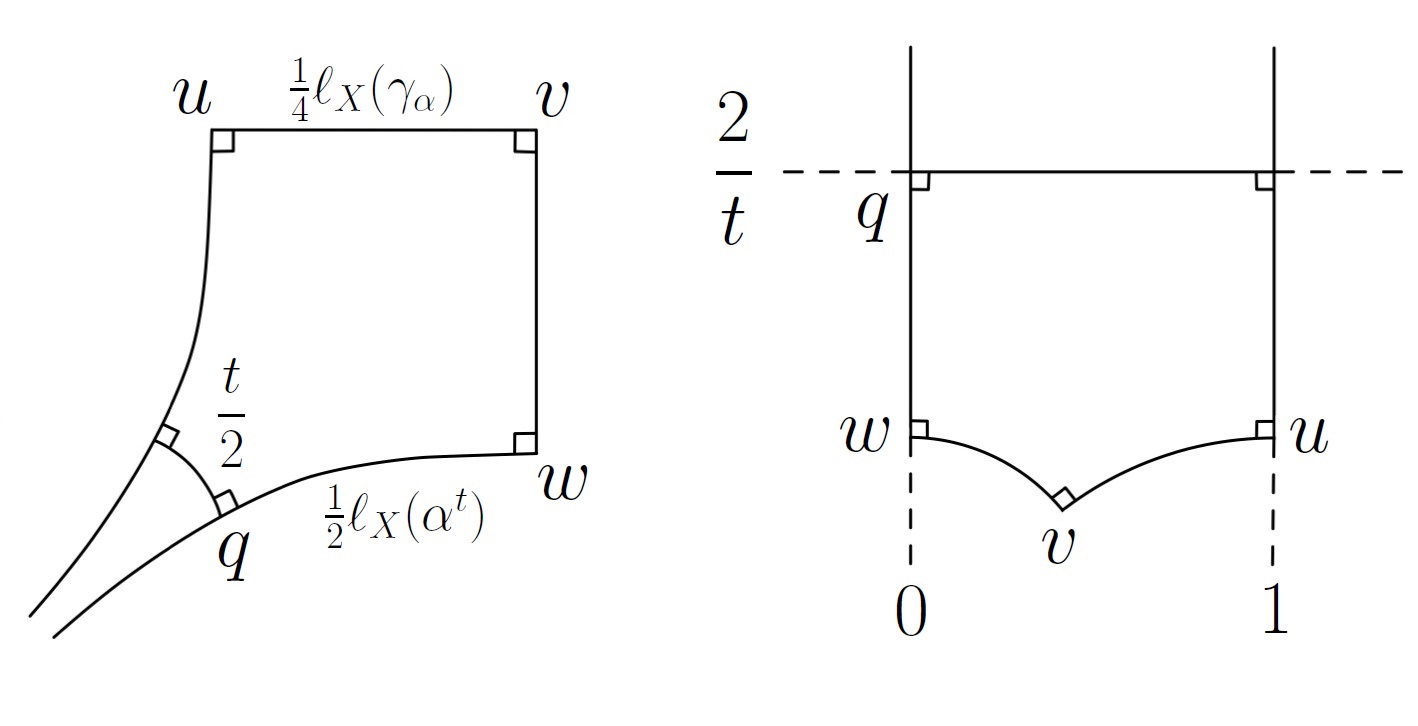}
        \caption{One of the quadrilaterals acquired from cutting $P$ (left), and the same quadrilateral in the upper-half plane model after normalising (right).}
        \label{fig:quadrilaterals}
    \end{figure}
    
    \newpage
    
    This means that $w=\frac{2}{t}i\ee^{-\frac{1}{2}\ell_X(\alpha^t)}$, and a computation shows that $\ell_X(\gamma_\alpha)=4\cosh^{-1}(\frac{t}{2}\ee^{\frac{1}{2}\ell_X(\alpha^t)})$. Hence, the difference $\ell_X(\gamma_\alpha)-2\ell_X(\alpha^t)$ can be written as \[4\cosh^{-1}\Big(\frac{t}{2}\ee^{\frac{1}{2}\ell_X(\alpha^t)}\Big)-2\ell_X(\alpha^t).\]
    The function $E_t(\ell)=4\cosh^{-1}(\frac{t}{2}\ee^{\frac{1}{2}\ell})-2\ell$ is continuous on $[m_t,\infty)$, where $m_t\coloneqq2\ln(\frac{2}{t})$ is a lower bound on the length of $\alpha^t$, and $\lim_{\ell\to\infty}E_t(\ell)=4\ln(t)$. It follows that there exists $C_1(t)>0$ such that
    \begin{equation}
    |\ell_X(\gamma_{\alpha})-2\ell_X(\alpha^t)|\leq C_1(t).
    \label{eq:arc length like curve length}
    \end{equation}
    
    Now suppose that $t>t_\alpha$. Note that $\alpha^t$ is contained in $\alpha^{t_\alpha}$, and so $\ell^t_X(\alpha)<\ell^{t_\alpha}_X(\alpha)$.
    Consider $\alpha^{t_\alpha}\setminus\alpha^t$, which lies in $\mathscr{H}_t\setminus\mathscr{H}_{t_\alpha}$. Exactly two of the components of $\alpha^{t_\alpha}\setminus\alpha^t$ are simple geodesic arcs from $\de\mathscr{H}_t$ to $\de\mathscr{H}_{t_\alpha}$ which meet each boundary orthogonally. Hence, these two components each have length $\ln(\frac{t}{t_\alpha})\leq\ln(\frac{1}{t_\alpha})$. The other components, if any, are returning segments in $\mathscr{H}_t$ with both endpoints on $\de\mathscr{H}_t$. Let $\beta$ be some such segment of $\alpha$, and let $d=\iota(\alpha,\alpha)$. Then we must have $\iota(\beta,\beta)\leq d$. Thus, as $t\leq1$, $\ell_X(\beta)\leq B(d)$ where $B$ is given by Lemma \ref{basmajian}. As $B$ only depends on $d$, this holds for any such segment. Now we need to show that there are only finitely many segments of $\alpha$ in $\mathscr{H}_t\setminus\mathscr{H}_{t_\alpha}$. From Lemma \ref{lem:beta self-intersects at least once}, we have that the self-intersection number of each segment is at least 1, and indeed the sum of the self-intersection numbers of these segments is at most $d$. Hence, $\alpha\cap(\mathscr{H}_t\setminus\mathscr{H}_{t_\alpha})$ has at most $d+2$ components and so
    \[\ell^{t_\alpha}_X(\alpha)-\ell^t_X(\alpha)\leq d B(d)+2\ln\Big(\frac{t}{t_\alpha}\Big)\leq d B(d)+2\ln\Big(\frac{1}{t_\alpha}\Big).\]
    Note that by Lemma \ref{lem:arcs give td}, $t_\alpha$ depends only on $d=\iota(\alpha,\alpha)$. Thus there exists some $C_2\big(\iota(\alpha,\alpha)\big)>0$ such that
    \begin{equation*}
    |\ell_X^{t_\alpha}(\alpha)-\ell_X^t(\alpha)|\leq C_2\big(\iota(\alpha,\alpha)\big).
    \end{equation*}
    
    Now by applying \eqref{eq:arc length like curve length} to $t_\alpha$, we have that
    $|\ell_X(\gamma_{\alpha})-2\ell_X(\alpha^{t_\alpha})|\leq C_1(t_\alpha),$
    and thus we can write
    \[|\ell_X(\gamma_\alpha)-2\ell_X(\alpha^t)|\leq C_1(t_\alpha)+2C_2\big(\iota(\alpha,\alpha)\big).\]
    
    Therefore, for any $t\leq1$,
    \[|\ell_X(\gamma_\alpha)-2\ell_X(\alpha^t)|\leq C\big(\iota(\alpha,\alpha),t\big)\]
    where
    \[C\big(\iota(\alpha,\alpha),t\big)=
    \begin{cases}
        C_1(t) & \text{if} \ t\leq t_\alpha, \\
        C_1(t_\alpha)+2C_2\big(\iota(\alpha,\alpha)\big) & \text{if} \ t>t_\alpha.
    \end{cases}
    \]
\end{proof}
\vspace{-\baselineskip}

\medskip
The fact that $J$ is $\PMod(S)$-equivariant holds by an argument analogous to the proof of Lemma \ref{lem:I equivariant}. That is, for any infinite arc $\alpha$ and any $\varphi\in\PMod(S)$,
\begin{equation}
    J(\varphi\cdot\alpha)=\varphi\cdot J(\alpha).
    \label{eq:J equivariant}
\end{equation}

For any infinite arc $\alpha_0$, $J_{\alpha_0}\colon\PMod(S)\cdot\alpha_0\to\PMod(S)\cdot\gamma_{\alpha_0}$ is the restriction of $J$ to $\PMod(S)\cdot\alpha_0$. Given any curve $\gamma$ of type $\gamma_{\alpha_0}$ and a fixed value of $t$, there are only finitely many arcs $\alpha$ of type $\alpha_0$ such that $\ell(\alpha^t)\leq\frac{1}{2}\ell_X(\gamma)+\frac{1}{2}C\big(\iota(\alpha_0,\alpha_0),t\big)$. By Lemma \ref{lem:bound infinite}, this means that there are only finitely many arcs $\alpha$ of type $\alpha_0$ such that $\gamma_\alpha=\gamma$. Using this together with \eqref{eq:J equivariant}, Proposition \ref{prop:k to 1} holds for $J_{\alpha_0}$ by an analogous argument.

\begin{prop}
    Let $\alpha_0$ be an infinite arc. Then there exists some $k=k(\alpha_0)$ such that $J_{\alpha_0}$ is surjective and $k$-to-1.
    \label{prop:k to 1 infinite}
\end{prop}

Armed with this, we can follow the argument from the proof of Theorem \ref{th:main compact} to prove Theorem \ref{th:main infinite}.

\begin{proof}[Proof of Theorem $\ref{th:main infinite}$]
    Let $\alpha_0$ be an infinite arc, and fix some positive $t\leq1$. Let $\gamma_{\alpha_0}$ be the curve associated to $\alpha_0$, as defined above. Using the same argument as in the proof of Theorem \ref{th:main compact}, replacing Corollary \ref{cor:bound compact multi} and Corollary \ref{corol:k-to-1 multi} with Lemma \ref{lem:bound infinite} and Proposition \ref{prop:k to 1 infinite}, we have
    \begin{align*}
        \lim_{L\to\infty}\frac{|\{\alpha \ \textup{of type} \ \alpha_0\mid\ell_X^t(\alpha)\leq L\}|}{L^{6g-6+2(n+p)}}&=k\cdot2^{6g-6+2(n+p)}\lim_{L\to\infty}\frac{|\{\gamma \ \textup{of type} \ \gamma_{\alpha_0}\mid\ell_X(\gamma)\leq L\}|}{L^{6g-6+2(n+p)}} \\
        &=k\cdot2^{6g-6+2(n+p)}\frac{\cc(\gamma_{\alpha_0})}{\bb_{g,n+p}}\mthu(\{\ell_X(\cdot)\leq1\}) \\
        &=\frac{\cc(\alpha_0)}{\bb_{g,n+p}}\mthu(\{\ell_X(\cdot)\leq1\})
    \end{align*}
    where $\cc(\alpha_0)=k\cdot2^{6g-6+2(n+p)}\cc(\gamma_{\alpha_0})$, $k$ is as in Proposition \ref{prop:k to 1 infinite}, and $\cc(\gamma_{\alpha_0})$ is as in \eqref{eq:th mirz}.
\end{proof}

\begin{remark}
    As previously mentioned, Theorem \ref{th:main infinite} holds when we replace the $t$-length by the truncated length $\ell_X^{Tr}$. This can be seen by applying Theorem \ref{th:main infinite} in the case that $t=t_{\alpha_0}$, and using the bound on the difference in the $t_{\alpha_0}$-length and the truncated length from \eqref{eq: truncated length bound}.
\end{remark}

\bibliography{myrefs}{}

\begin{thebibliography}{10}

\bibitem{Bas93}
Ara Basmajian.
\newblock The orthogonal spectrum of a hyperbolic manifold.
\newblock {\em Amer. J. Math.}, 115(5):1139--1159, 1993.

\bibitem{Bas13}
Ara Basmajian.
\newblock Universal length bounds for non-simple closed geodesics on hyperbolic
  surfaces.
\newblock {\em J. Topol.}, 6(2):513--524, 2013.

\bibitem{BPT20}
Ara Basmajian, Hugo Parlier, and Ser~Peow Tan.
\newblock Prime orthogeodesics, concave cores and families of identities on
  hyperbolic surfaces.
\newblock Preprint, {\tt arXiv:2006.04872}, 2020.

\bibitem{Buser}
Peter Buser.
\newblock {\em Geometry and spectra of compact {R}iemann surfaces}.
\newblock Modern Birkh\"{a}user Classics. Birkh\"{a}user Boston, Ltd., Boston,
  MA, 2010.
\newblock Reprint of the 1992 edition.

\bibitem{EPS16}
Viveka Erlandsson, Hugo Parlier, and Juan Souto.
\newblock Counting curves, and the stable length of currents.
\newblock {\em J. Eur. Math. Soc. (JEMS)}, 22(6):1675--1702, 2020.

\bibitem{ES19-RA}
Viveka Erlandsson and Juan Souto.
\newblock Mirzakhani's {C}urve {C}ounting.
\newblock Preprint, {\tt arXiv:1904.05091}, 2019.

\bibitem{ES20-BOOK}
Viveka Erlandsson and Juan Souto.
\newblock {\em Geodesic currents and Mirzakhani's curve counting}.
\newblock Book in preparation. 2020.

\bibitem{Primer}
Benson Farb and Dan Margalit.
\newblock {\em A primer on mapping class groups}, volume~49 of {\em Princeton
  Mathematical Series}.
\newblock Princeton University Press, Princeton, NJ, 2012.

\bibitem{Hatcher88}
A.~E. Hatcher.
\newblock Measured lamination spaces for surfaces, from the topological
  viewpoint.
\newblock {\em Topology Appl.}, 30(1):63--88, 1988.

\bibitem{He18}
Yan~Mary He.
\newblock Prime number theorems for {B}asmajian-type identities.
\newblock Preprint, {\tt arXiv:1811.05367}, 2018.

\bibitem{Huber}
Heinz Huber.
\newblock Zur analytischen {T}heorie hyperbolischen {R}aumformen und
  {B}ewegungsgruppen.
\newblock {\em Math. Ann.}, 138:1--26, 1959.

\bibitem{Margulis}
G.~A. Margulis.
\newblock Certain applications of ergodic theory to the investigation of
  manifolds of negative curvature.
\newblock {\em Funkcional. Anal. i Prilo\v{z}en.}, 3(4):89--90, 1969.

\bibitem{Mcshane}
Greg McShane.
\newblock {\em A remarkable identity for lengths of curves}.
\newblock ProQuest LLC, Ann Arbor, MI, 1991.
\newblock Thesis (Ph.D.)--University of Warwick (United Kingdom).

\bibitem{Mirz08}
Maryam Mirzakhani.
\newblock Growth of the number of simple closed geodesics on hyperbolic
  surfaces.
\newblock {\em Ann. of Math. (2)}, 168(1):97--125, 2008.

\bibitem{Mirz16}
Maryam Mirzakhani.
\newblock Counting {M}apping {C}lass group orbits on hyperbolic surfaces.
\newblock Preprint, {\tt arXiv:1601.03342}, 2016.

\bibitem{PP}
Jouni Parkkonen and Fr\'{e}d\'{e}ric Paulin.
\newblock Counting common perpendicular arcs in negative curvature.
\newblock {\em Ergodic Theory Dynam. Systems}, 37(3):900--938, 2017.

\bibitem{Parlier20}
Hugo Parlier.
\newblock Geodesic and orthogeodesic identities on hyperbolic surfaces.
\newblock Preprint, {\tt arXiv:2004.09078}, 2020.

\bibitem{Penner87}
R.~C. Penner.
\newblock The decorated {T}eichm\"{u}ller space of punctured surfaces.
\newblock {\em Comm. Math. Phys.}, 113(2):299--339, 1987.

\bibitem{Pennernotes}
R.~C. Penner.
\newblock Lambda {L}engths.
\newblock {\em Aarhus University, Lecture Notes}, online at
  http://www.ctqm.au.dk/research/MCS/lambdalengths.pdf, 2006.

\bibitem{PenHar}
R.~C. Penner and J.~L. Harer.
\newblock {\em Combinatorics of train tracks}, volume 125 of {\em Annals of
  Mathematics Studies}.
\newblock Princeton University Press, Princeton, NJ, 1992.

\bibitem{Thurstonnotes}
W.~P. Thurston.
\newblock The {G}eometry and {T}opology of {T}hree-{M}anifolds.
\newblock {\em Princeton University, Lecture Notes}, online at
  http://www.msri.org/publications/books/gt3m, 1986.

\end{thebibliography}
\bibliographystyle{plain}

\end{document}